\documentclass{article}%
\usepackage{amssymb}
\usepackage{amsfonts}
\usepackage{graphicx}
\usepackage{amsmath}%
\providecommand{\U}[1]{\protect\rule{.1in}{.1in}}
\newtheorem{theorem}{Theorem}

\newtheorem{corollary}[theorem]{Corollary}

\newtheorem{definition}[theorem]{Definition}

\newtheorem{lemma}[theorem]{Lemma}

\newtheorem{proposition}[theorem]{Proposition}
\newtheorem{remark}[theorem]{Remark}

\newenvironment{proof}[1][Proof]{\textbf{#1.} }{\ \rule{0.5em}{0.5em}}
\begin{document}

\begin{center}
{\LARGE Properly outer and strictly outer actions of finite groups on prime
C*-algebras}

\bigskip

\bigskip

Costel Peligrad
\end{center}

\bigskip

Department of Mathematical Sciences, University of Cincinnati, PO Box 210025,
Cincinnati, OH 45221-0025, USA. E-mail address: costel.peligrad@uc.edu



\bigskip

\textbf{ABSTRACT. An action of a compact, in particular finite group on a
C*-algebra is called properly outer if no automorphism of the group that is
distinct from identity is implemented by a unitary element of the algebra of
local multipliers of the C*-algebra. In this paper I define the notion of
strictly outer action (similar to the definition for von Neumann factors in
[11]) and prove that for finite groups and prime C*-algebras, it is equivalent
to the proper outerness of the action. For finite abelian groups this is
equivalent to other relevant properties of the action. }

\bigskip

\textbf{ }

\section{Introduction}

\bigskip

An action $\alpha$ of a locally compact group or a quantum group, $G,$ on a
von Neumann factor $M$ \ is said to be strictly outer when the relative
commutant of the factor in the crossed product is trivial. The action is
called minimal if the inclusion $M^{\alpha}\subset M$ is irreducible, that is
the relative commutant $(M^{\alpha})^{\prime}\cap M$\ of the fixed point
algebra $M^{\alpha}$\ in $M$\ is trivial. For integrable actions of quantum
groups on factors an action is minimal if and only if it is strictly outer
[11]. It can be shown that, in particular, for finite groups, an action
$\alpha$\ on a factor $M$ is strictly outer if and only if it is properly
outer (that is, no $\alpha_{g},g\neq e$\ is implemented by a unitary element
of $M$). In the next Section, I define strictly outer actions of finite groups
on C*-algebras and, in Section 3.1., prove that an action of a finite group,
$G,$ on a prime C*-algebra, $A,$ is strictly outer if and only if it is
properly outer (that is, no $\alpha_{g},g\neq e,$\ is implemented by a unitary
element in the algebra of local multipliers of $A$). In Section 3.2. I study
actions of finite abelian groups on prime C*-algebras and state some
equivalent conditions to proper and therefore strict outerness. These results
extend and complement some known results in the litterature ([4], [5], [7]).
In [2] similar results were obtained for compact abelian group actions on
separable prime C*-algebras with an additional condition. I do not require
separability of C*-algebras.

\bigskip

\section{Definitions and preliminary results}

If $A$\ is a C*-algebra, denote by $\mathcal{I}$\ the set of all closed
essential, two sided ideals of $A$. A two sided ideal $I\subset A$\ is called
essential if $\left\{  a\in A:aI=(0)\right\}  =(0).$ If $I\subset J$\ are two
such ideals then according to [7, Proposition 3.12.8.] we have the inclusion
$M(J)\subset M(I)$\ where $M(I),$\ respectively $M(J)$ are the multiplier
algebras of $I,$\ respectively $J.$\ The inductive limit of the system
$\left\{  M(I):I\in\mathcal{I}\right\}  $\ is called the algebra of local
multipliers of $A$\ and denoted $M_{loc}(A)$ [1]$.$ Originally, this algebra
was denoted $M^{\infty}(A)$\ and called the algebra of essential multipliers
of $A$\ [6]. Recall that a C*-algebra, $A,$ is called prime if every (two
sided) ideal of $A$\ is essential.

In the next lemma I collect some known results

\begin{lemma}
Let $A$\ be a C*-algebra. The following properties are equivalent\newline a)
$A$\ is prime.\newline b) The center of $M_{loc}(A)$\ consists of
scalars.\newline c) If $B$\ is a C*-subalgebra of $M_{loc}(A)$\ such that
$A\subset B\subset M_{loc}(A),$ then the center of $B$\ consists of scalars.
\end{lemma}

\begin{proof}
The equivalence of a) and b) follows from [2, Proposition 3.1.] (see also [6,
page 303, statements 6), 4) and 5)]). The implication b)$\Rightarrow$c)
follows from [1, Lemma 3.2.2.]. and the implication c)$\Rightarrow$b) is obvious.
\end{proof}

\bigskip

In the rest of this paper we will consider a C*-dynamical system
$(A,G,\alpha)$\ where $A$\ is a C*-algebra, $G$\ a finite (or compact) group
and $\alpha$\ a faithful action of $G$\ on $A,$\ that is for each $g\in
G,$\ $\alpha_{g}$\ is an automorphism of $A$\ such that $\alpha_{g_{1}g_{2}%
}=\alpha_{g_{1}}\alpha_{g_{2}}$ for all $g_{1},g_{2}\in G.$ We wil denote by
$A^{\alpha}=\left\{  a\in A:\alpha_{g}(a)=a,g\in G\right\}  $\ the fixed point
algebra of the action. The action $\alpha$\ is called faithful if $\alpha
_{g}\neq id$ for $g\neq e$\ where $e$\ is the unit of $G.$ If the group $G$ is
finite, the action $\alpha$ of $G$ on\ $A$\ has a natural extension to
$M_{loc}(A).$

\begin{lemma}
\bigskip Suppose $G$\ is a compact group. Then,\newline a) Every approximate
identity of $A^{\alpha}$\ is an approximate identity of $A,$\newline b)
$M(A)^{\alpha}=M(A^{\alpha}).$
\end{lemma}

\begin{proof}
\bigskip a) follows from the more general case of compact quantum group
actions [3, Lemma 2.7.]. To prove b), notice first that $M(A)^{\alpha}\subset
M(A^{\alpha}),$ then, using a) it follows that $M(A^{\alpha})\subset M(A),$ so
$M(A^{\alpha})\subset M(A)^{\alpha}$.
\end{proof}

\begin{lemma}
Let $(A,G,\alpha)$\ be a C*-dynamical system with $G$\ \ finite and
$A$\ prime. Then, \newline a) $M_{loc}(A)^{\alpha}\subset M_{loc}(A^{\alpha
}).$\newline b) If $x\in M_{loc}(A)$ commutes with $A^{\alpha}$\ then
$x$\ commutes with $M_{loc}(A)^{\alpha}.$\newline c) If $A^{\alpha}$\ is
prime, then the center of $M_{loc}(A)^{\alpha}$\ consists of scalars.\newline
d) $A^{\prime}\cap M_{loc}(A)$ consists of scalars.
\end{lemma}

\begin{proof}
a) Since $G$\ is finite, every essential ideal $J\subset A$\ contains the
$\alpha$-invariant essential ideal $I=\cap_{g\in G}\alpha_{g}(J).$ Therefore,
$M_{loc}(A)$\ is the inductive limit of $\left\{  M(I):I\subset A\text{
}\alpha\text{-invariant essential ideal}\right\}  .$If $I\subset A$\ is a
closed two sided $\alpha$-invariant essential ideal of $A,$\ it is obvious
that $I^{\alpha}$\ is an essential ideal of $A^{\alpha}$. From Lemma 2 b)
applied to $I$\ instead of $A$\ it follows that $M(I)^{\alpha}=M(I^{\alpha
})\subset M_{loc}(A^{\alpha})$ and, therefore, $M_{loc}(A)^{\alpha}\subset
M_{loc}(A^{\alpha}).$\newline To prove b), let $I\subset A$\ \ be an $\alpha
$-invariant essential ideal. Then $I^{\alpha}$\ is ideal of $A^{\alpha}$\ so
$x$\ commutes with $I^{\alpha}.$ If $m\in M(I^{\alpha})$ and $a\in I^{\alpha}%
$\ we have%
\[
xam=amx.
\]
and%
\[
xam=axm.
\]
Therefore%
\[
a(xm-mx)=0.
\]
so $I^{\alpha}(xm-mx)=\left\{  0\right\}  .$ Applying Lemma 2 a) to
$I$\ instead of $A$,\ it follows that $I(xm-mx)=\left\{  0\right\}  .$ From
[1, \ Lemma 2.3.3.ii)] we get $xm-mx=0.$ Applying Lemma 2 b) to $I,$ we
have\ $M(I^{\alpha})=M(I)^{\alpha},$ so $x$\ commutes with $M(I)^{\alpha}$ for
every $\alpha$-invariant essential ideal of $A.$ Therefore, $x$\ commutes with
$\underline{lim}_{I}M(I)^{\alpha}=M_{loc}(A)^{\alpha}.$\newline c) If
$A^{\alpha}$\ is prime, from Lemma 1 b) applied to $A^{\alpha}$\ instead of
$A$\ it follows that the center of $M_{loc}(A^{\alpha})$\ consists of scalars.
Applying a) it follows that $M_{loc}(A)^{\alpha}\subset M_{loc}(A^{\alpha}).$
From Lemma 1 c) applied to $A^{\alpha}$\ instead of $A$\ it follows that the
center of $M_{loc}(A)^{\alpha}$\ consists of scalars.\newline d) Follows from
part b) for trivial $\alpha$.
\end{proof}

\bigskip

The cross product $A\times_{\alpha}G$ of $A$\ by a finite group, $G,$\ is
defined as the C*-completion of the algebra of all functions $f:G\rightarrow
A$ with the following operations:%
\[
(f\cdot h)(g)=\sum_{p\in G}f(p)\alpha_{p}(h(p^{-1}g).
\]%
\[
f^{\ast}(g)=\alpha_{g}(f(g^{-1})^{\ast}.
\]
As $G$\ is a finite group, each $f\in A\times_{\alpha}G$\ can be written
$f=\sum a_{g}\delta_{g}$\ where $a_{g}\in A$\ and $\delta_{g}$\ is the
function on $G$ with values in the multiplier algebra $M(A)$\ of $A$\ such
that $\delta_{g}(g)=1$\ and $\delta_{g}(p)=0$\ for $p\neq g.$

If $A$\ is a C*-algebra then, $A\subset A^{\prime\prime}\subset B(H_{u}%
)$\ where $A^{^{\prime\prime}}$\ is enveloping von Neumann algebra of $A$\ in
$B(H_{u})$ and $H_{u}$ is the Hilbert space of the universal representation of
$A$ [7, Section 3.7.]$.$Denote by $id$ the identity representation of $A\ $in
$B(H_{u}).$ If $G$\ is a finite group, denote by $\lambda$\ the left regular
representation of $G$ on $l^{2}(G)$. Denote by ($\widetilde{id},\lambda
,l^{2}(G,H_{u})),$the covariant representation of $(A,G,\alpha)$ on
$l^{2}(G,H_{u}),$ where%
\[
\widetilde{id}(a)(h)=\alpha_{h^{-1}}(a),\lambda_{g}(\xi)(h)=\xi(g^{-1}h).
\]
Consider the regular representation, $\widetilde{id}\times\lambda,$ of
$A\times_{\alpha}G$ induced by ($\widetilde{id},\lambda,l^{2}(G,H_{u})):$%
\[
(\widetilde{id}\times\lambda)(f)(p)=%
{\displaystyle\sum\limits_{g\in G}}
\alpha_{p^{-1}}(f(g))\xi(g^{-1}p).
\]
Then, $\widetilde{id}\times\lambda$\ is a faithful non degenerate
representation of $A\times_{\alpha}G$\ on $l^{2}(G,H_{u}$) [7, Section 7.7.].

\begin{remark}
\ $A\times_{\alpha}G\subset M(A)\times_{\alpha}G\subset A^{^{\prime\prime}%
}\times_{\alpha}G$ where $A^{^{\prime\prime}}\times_{\alpha}G$ is the W*-cross
product of $A^{^{\prime\prime}}$\ by the double dual action $\alpha
^{^{\prime\prime}}$ of $\alpha.$
\end{remark}

\bigskip I will also use the following

\begin{lemma}
Suppose $G\ $is finite. Then, $M_{loc}(A)\subset M_{loc}(A\times_{\alpha}G).$
\end{lemma}

\begin{proof}
According to the proof of Lemma 3. a), $M_{loc}(A)$\ is the inductive limit of
$\left\{  M(I):I\subset A\text{ }\alpha\text{-invariant essential
ideal}\right\}  $. A standard argument, using that $G$\ is finite, shows that
for every $\alpha$-invariant essential ideal $I\subset A,\ I\times_{\alpha}%
G$\ is an essential ideal of $A\times_{\alpha}G.$ Since $M(I)\subset
M(I\times_{\alpha}G)\subset M_{loc}(A\times_{\alpha}G)$, it follows that
$M_{loc}(A)\subset M_{loc}(A\times_{\alpha}G).$
\end{proof}

\bigskip In the following definitions we assume that the group $G$\ is finite.

\begin{definition}
a) If $(A,G,\alpha)$\ is a C*-dynamical system with $A$\ prime, the action
$\alpha$\ is called properly outer if no $\alpha_{g},g\neq e$\ is implemented
by a unitary element in $M_{loc}(A).$\newline b) The action is called strictly
outer if $A^{\prime}\cap M_{loc}(A\times_{\alpha}G)=%
\mathbb{C}
I,$\ where $A^{\prime}$\ is the commutant of $A.$
\end{definition}

\bigskip

In [2, Proposition 3.2.] the authors prove that two other statements are
equivalent with the above Definition 6. part a) for separable C*-algebras. For
not necessarily separable C*-algebras, I use the statement given in the above
Definition 6. a). Actually, as proven in [2], the other two statements imply
the Definition 6. a) for not necessarily separable C*-algebras.

\bigskip

\begin{center}
Now, let $(A,G,\alpha)$\ be a C*-dynamical system with $G$\ \ compact abelian.
Denote by $\widehat{G}$ the dual of $G.$ If $\gamma\in\widehat{G}$\ let
$A_{\gamma}=\left\{  a\in A:\alpha_{g}(a)=\left\langle g,\gamma\right\rangle
a,g\in G\right\}  $ be the corresponding spectral subspace of $A.$ Denote by
$sp(\alpha)$\ the Arveson spectrum of $\alpha$\ and by $\Gamma(\alpha)$\ the
Connes spectrum of $\alpha$:
\[
sp(\alpha)=\left\{  \gamma\in\widehat{G}:A_{\gamma}\neq\left\{  0\right\}
\right\}  .
\]
and, if $\mathcal{H}$\ is the set of all non zero $\alpha$-invariant
hereditary subalgebras of $A$
\[
\Gamma(\alpha)=\cap\left\{  sp(\alpha|_{B}:B\in\mathcal{H}\right\}  .
\]
I will list some standard facts about the Arveson spectrum of actions of
compact abelian groups (in particular finite abelian groups) that will be used
in Section 3.2.:
\end{center}

1) The linear span of $\cup\left\{  A_{\gamma}:\gamma\in Sp(\alpha)\right\}
$\ is dense in $A.$ Therefore, if $\left\langle t,\gamma\right\rangle =1$ for
every $\gamma\in Sp(\alpha),$ then $\alpha_{t}$\ is the identity automorphism.

2) If $\gamma_{1},\gamma_{2}\in Sp(\alpha),$ then $A_{\gamma_{1}}A_{\gamma
_{2}}\subset A_{\gamma_{1}+\gamma_{2}}.$ Here, $A_{\gamma_{1}}A_{\gamma_{2}}$
denotes the linear span of $\left\{  ab:a\in A_{\gamma_{1}},b\in A_{\gamma
_{2}}\right\}  .$

3) If $\gamma\in Sp(\alpha),$ $A_{\gamma}A_{\gamma}^{\ast}$ and $A_{\gamma
}^{\ast}A_{\gamma}$\ are ideals of $A^{\alpha}$. Here $A_{\gamma}A_{\gamma
}^{\ast}$ denotes the linear span of $\left\{  ab^{\ast}:a,b\in A_{\gamma
}\right\}  .$

4) If $A^{\alpha}$\ is prime, $Sp(\alpha)$\ is a subgroup of $\widehat{G}%
.$This follows from 2) and 3).

5) If $A^{\alpha}$ is prime and the action $\alpha$\ is faithful (that is
$\alpha_{t}\neq id$\ for $t\neq0$), then $Sp(\alpha)=\widehat{G}$. This
follows from 4) and 1).

The following result was proven by Olesen for simple C*-algebras, $A$ and
actions for which $sp(\alpha)/\Gamma(\alpha)$\ is compact in $\widehat{G}%
/\Gamma(\alpha)$ in [4, Theorem 4.2.] (see also [7, Theorem 8.9.7.]). If
($A,G,\alpha)$ is a C*-dynamical system, $A$\ is called $G$-prime if every non
zero $\alpha$-invariant ideal of $A$ is an essential ideal.

\bigskip

\begin{proposition}
Let ($A,G,\alpha)$ be a C*-dynamical system with $A$ $G$-prime and
$G$\ compact abelian such that $\widehat{G}/\Gamma(\alpha)$\ is finite (in
particular $G$\ finite group). If \ $t_{0}\in G$\ is such that $\left\langle
t_{0},\gamma\right\rangle =1$\ \ for every $\gamma\in\Gamma(\alpha)$\ then
there exists an $\alpha$-invariant (essential) ideal $J\subset A$\ and a
unitary operator $u\in M(J)^{\alpha}$\ such that $\alpha_{t_{0}}=Adu.$
\end{proposition}

\begin{proof}
\bigskip Applying [7. Proposition 8.8.7.] for $A$ $G$-prime and $G$\ compact
with $\widehat{G}/\Gamma(\alpha)$\ finite, it follows that there exists a non
zero $\alpha-$invariant hereditary C*-subalgebra, $C\subset A$\ such that
$sp(\alpha|_{C})=\Gamma(\alpha).$Therefore, $\alpha_{t_{0}}|_{C}=id.$ Let
$J=\overline{ACA}$\ be the closed two sided ideal of $A$\ generated by $C.$
Since $C$\ is $\alpha$-invariant and $A$\ is $G$-prime, it follows that
$J$\ is an esseential $\alpha$-invariant ideal of $A.$ Using an approximate
identity of $A$\ it follows that $C\subset J\subset J^{\prime\prime}$.

Let $p$\ be the open projection\ corresponding to $C$\ in $A^{^{\prime\prime}%
}$. Then $p\in J^{\prime\prime}$. \ Denote by $1_{J}$ the unit of
$J^{^{\prime\prime}}$ ($1_{J}$ is the open projection in $A^{^{\prime\prime}%
}\ $corresponding.to $J).$ Notice that the central support $c(p)$\ of $p$\ in
$J^{^{\prime\prime}}$ equals $1_{J}.$\ Indeed if $q\geqslant p$\ is a central
projection in $J^{^{\prime\prime}},$ then $qJ^{^{\prime\prime}}$ is an ideal
of $J"$\ that contains $C,$\ so $q=1_{J}$.

Applying [7, Lemma 8.9.1.] to $M=J^{^{\prime\prime}}$ and $w=p,$ it follows
that there exists a unique unitary $u\in J^{^{\prime\prime}}\ $such that
$up=p$\ and $\alpha_{t_{0}}^{\prime\prime}|_{J^{^{\prime\prime}}%
}=Adu|_{J^{\prime\prime}}.$ Since $\alpha_{t_{0}}^{^{\prime\prime}%
}(p)=p=upu^{\ast}$\ \ we have $up=pu=p.$ Since $\alpha_{t_{0}}^{^{\prime
\prime}}|_{J^{^{^{\prime\prime}}\prime}}=Adu|_{J^{^{\prime\prime}}}$ and
$M(J)\subset J^{^{\prime\prime}},$\ it follows that $\alpha_{t_{0}}%
|_{M(J)}=Adu|_{M(J)}.$ Also, since $G$ is abelian, $\alpha_{t}(p)=p,$ for
every $t\in G,$ and $u$\ is unique with the listed properties, it follows
that, $\alpha_{t}(u)=u$ for every $t\in G.$

We will prove next that $u\in M(J).$\ Let $J_{0}=\left\{  x\in J:ux\in
J\right\}  .$\ Clearly $J_{0}$\ is a closed right ideal of $J.$\ Now, if $y\in
J,$\ and $x\in J_{0}$\ we have $u%
\operatorname{y}%
x=u%
\operatorname{y}%
u^{\ast}ux$\ =$\alpha_{t_{0}}(y)ux\in J,$ so $J_{0}$\ \ is also a left ideal.
Notice that $C\subset J_{0}.$ Indeed, if $x\in C,$ we have $ux=upx=px=x\in
C\subset J,$ so $x\in J_{0}.$ Therefore, $C\subset J_{0}$\ $\subset J$\ .
Since $J$ is the (two sided) ideal generated by $C$\ it follows that
$J_{0}=J,$ so $u$\ is a left multiplier of $J.$\ To prove that $u$\ is also a
right multiplier of $J$, notice that if $x\in J$\ we have $xu=uu^{\ast
}xu=u\alpha_{t_{0}}^{-1}(x)\in J.$ Since $M(A)\subset M(J)$ (in $M_{loc}(A)),$
it follows that $\alpha_{t_{0}}=Adu.$
\end{proof}

\begin{corollary}
Let ($A,G,\alpha)$ be a C*-dynamical system with $A$ $G$-prime and $G$\ finite
abelian. If \ $t_{0}\in G$\ is such that $\left\langle t_{0},\gamma
\right\rangle =1$\ \ for every $\gamma\in\Gamma(\alpha)$\ then $\alpha_{t_{0}%
}$\ is inner in $M_{loc}(A)^{\alpha}.$
\end{corollary}

\begin{proof}
\bigskip Proposition 7 shows that $\alpha_{t_{0}}=Adu,u\in M(J)^{\alpha
}\subset M_{loc}(J)$. By [6, page 303, statement 3)], since $J$\ is an
essential ideal of $A,$\ \ $M_{loc}(J)=M_{loc}(A),$ so $\alpha_{t_{0}}$\ is
inner in $M_{loc}(A)^{\alpha}.$
\end{proof}

\bigskip I will close this section with the following

\begin{theorem}
Let $(A,G,\alpha)$ be a C*-dynamical system with $G$\ finite abelian, $\alpha
$\ faithful and $A$ $G$-prime. Then the following statements are
equivalent:\newline a) $Sp$($\alpha)=\Gamma(\alpha).$\newline b) $A^{\alpha}%
$\ is prime.\newline c) The center of $M_{loc}(A)^{\alpha}$ consists of
scalars.\newline d) No $\alpha_{t}\neq id$\ is implemented by a unitary in
$M_{loc}(A)^{\alpha}.$
\end{theorem}

\begin{proof}
$a)\Leftrightarrow b)$ This is [7 Theorem 8.10.4.].\newline$b)\Rightarrow c)$
Since $A^{\alpha}$\ is prime, from Lemma 1, a)$\Leftrightarrow$b$),$ with
$A^{\alpha}$\ instead of $A$ it follows that the center of $M_{loc}(A^{\alpha
})$\ is $%
\mathbb{C}
I.$\ Applying Lemma 1, b)$\Leftrightarrow$c) to $A^{\alpha}$\ instead of $A$,
it follows that the center of $M_{loc}(A)^{\alpha}\subset M_{loc}(A^{\alpha})$
is $%
\mathbb{C}
I.$\newline$c)\Rightarrow d)$ If for some $t_{0}\in G,$\ $\alpha_{t_{0}}$\ is
implemented by a unitary $u\in M_{loc}(A)^{\alpha},$ then $u$ commutes with
$A^{\alpha}$. $\ $From Lemma 3 b) it follows that $u\ $commutes with
$M_{loc}(A)^{\alpha},$ so $u$ belongs to the center of $M_{loc}(A)^{\alpha},$
hence $u$ is a scalar and $\alpha_{t_{0}}=id.$\newline$d)\Rightarrow a)$
Suppose $Sp$($\alpha)\neq\Gamma(\alpha).$\ Since $\Gamma(\alpha)$ is a
subgroup of $\widehat{G}$, there exists $t_{0}\in G$\ such that $\left\langle
t_{0},\gamma\right\rangle =1$\ \ for every $\gamma\in\Gamma(\alpha)$ and
$\left\langle t_{0},\gamma\right\rangle \neq1$ for some $\gamma\in
Sp(\alpha).$ From Corollary 8 it follows that $\alpha_{t_{0}}\neq id$ and
$\alpha_{t_{0}}$\ is implemented by a unitary in $M_{loc}(A)^{\alpha},$
contradiction$.$
\end{proof}

\bigskip

\section{Outer and strictly outer actions of finite groups}

\subsection{Finite non abelian groups}

\bigskip

\begin{lemma}
Let $(A,G,\alpha)$ be \ a C*-dynamical system with $G$\ finite.\newline a)
Then $M(A\times_{\alpha}G)=M(A)\times_{\alpha}G.$\newline\ b) If A is prime
and $\alpha$ is properly outer, then $M_{loc}(A\times_{\alpha}G)$ is
isomorphic to $M_{loc}(A)\times_{\alpha}G.$ \ 
\end{lemma}

\begin{proof}
a) It is obvious that ($\widetilde{id}\times\lambda)(A\times_{\alpha}G)$ is
included in the W*-crossed product ($\widetilde{id}\times\lambda)(A^{"}%
\times_{\alpha^{"}}G).$\ Since $\widetilde{id}\times\lambda$ is a faithful and
non degenerate representation of $A\times_{\alpha}G,$ the multiplier algebra
$M(A\times_{\alpha}G)$ is isometrically isomorphic with the algebra of
multipliers of ($\widetilde{id}\times\lambda)(A\times_{\alpha}G)$ in
(($\widetilde{id}\times\lambda)(A^{"}\times_{\alpha^{"}}G))\subset
B(l^{2}(G,H_{u}))$ [7, Prop. 3.12.3]. Let $b=%
{\displaystyle\sum\limits_{g\in G}}
b_{g}\delta_{g}\in A^{"}\times_{\alpha^{"}}G$\ be a multiplier of
$A\times_{\alpha}G.$ If $a\in A$\ is arbitrary, $\widetilde{a}=a\delta_{e}\in
A\times_{\alpha}G,$ where $e$\ is the neutral elementof $G$ and $\cdot$\ is
the multiplication of $A^{"}\times_{\alpha^{"}}G,$\ we have%
\[
(b\cdot\widetilde{a})_{g}=%
{\displaystyle\sum\limits_{p\in G}}
b_{p}\alpha_{p}(\widetilde{a}_{p^{-1}g})=b_{g}\alpha_{g}(a)\in A.
\]%
\[
(\widetilde{a}\cdot b)_{g}=%
{\displaystyle\sum\limits_{p\in G}}
\widetilde{a}_{p}\alpha_{p}(b_{p^{-1}g})=ab_{g}\in A.
\]
Therefore, $b_{g}\in M(A)$\ for every $g\in G,$\ so $b\in M(A)\times_{\alpha
}G$.

b) Let $I$\ be an $\alpha$-invariant essential ideal of $A.$\ Then, a simple
computation shows that $I\times_{\alpha}G$\ is an essential ideal of
$A\times_{\alpha}G$ [10, Proposition 2.2.]$.$ By a), $M(I\times_{\alpha
}G)=M(I)\times_{\alpha}G.$ If $J$\ is an essential ideal of $A\times_{\alpha
}G$\ then there exists an $\alpha$-invariant essential ideal $I$\ of $A$\ such
that $I\times_{\alpha}G\subset J,$ so $M(J)\subset M(I\times_{\alpha
}G)=M(I)\times_{\alpha}G.$ To prove that such an ideal $I$\ exists, notice
that, since $A$\ is prime and $\alpha$\ properly outer, the action $\alpha
$\ is, in particular, "purely outer" in the sense used by Rieffel in [10].
From [10 Theorem 1.1.] it follows that $I_{0}=J\cap A\neq\left\{  0\right\}
.$\ Since $J\subset A\times_{\alpha}G$\ is essential it follows that $I_{0}%
$\ is essential in $A.$\ Since $G$\ is finite, it follows that $I=\cap_{g\in
G}\alpha_{g}(I_{0})\subset I_{0}$\ is an essential $\alpha$-invariant ideal of
$A$ and $I\times_{\alpha}G\subset J.$\ Therefore, $M(J)\subset M(I\times
_{\alpha}G)$. According to a), $M(I\times_{\alpha}G)=M(I)\times_{\alpha}%
G.$\ Therefore%
\[
M(J)\subset M(I)\times_{\alpha}G\subset M_{loc}(A)\times_{\alpha}G.
\]
Hence%
\[
M_{loc}(A\times_{\alpha}G)\subset M_{loc}(A)\times_{\alpha}G.
\]
On the other hand,\ by Lemma 5, $M_{loc}(A)\subset M_{loc}(A\times_{\alpha
}G),$ so\
\[
M_{loc}(A)\times_{\alpha}G\subset M_{loc}(A\times_{\alpha}G).
\]

\end{proof}

\bigskip

\ \ \ \ \ \ \ \ \ \ \ 

\begin{theorem}
Let $(A,G,\alpha)$ be \ a C*-dynamical system with $G$\ finite, $A$\ prime and
$\alpha$\ faithful. Then $\alpha$\ is properly outer if \ and only if it is
strictly outer.
\end{theorem}

\begin{proof}
\ Suppose $\alpha$\ is properly outer and let $v=\sum_{h\in G}v(h)\delta
_{h}\in A^{\prime}\cap M_{loc}(A\times_{\alpha}G)(=M_{loc}(A)\times_{\alpha}G$
by Lemma 10) be a unitary element. Then, for $a\in A$ denote $\widetilde{a}%
=a\delta_{e}\in A\times_{\alpha}G\subset M_{loc}(A)\times_{\alpha}G$
($=M_{loc}(A\times_{\alpha}G)$ by Lemma 10.). We have%

\[
v\widetilde{a}=\widetilde{a}v\text{ \textit{in} }M_{loc}(A)\times_{\alpha
}G\text{ \textit{for all} }a\in A.
\]
Therefore%
\[
\sum_{h}v(h)\alpha_{h}(\widetilde{a}(h^{-1}g))=\sum\widetilde{a}(h)\alpha
_{h}(v(h^{-1}g)),g\in G.
\]
So%
\[
v(g)\alpha_{g}(a)=av(g)\text{ \textit{for all} }g\in G,a\in A.\text{ \ }%
(\ast)
\]
We will prove next that if $v(g)\neq0$ then $v(g)$\ is a scalar multiple of a
unitary element in $M_{loc}(A).$ If we denote\
\[
\widetilde{v}_{g_{0}}=v(g_{0})\delta_{g_{0}}\in M_{loc}(A)\times_{\alpha}G.
\]
then, since\ $v(g)\alpha_{g}(a)=av(g),g\in G$, we have $\widetilde{v}_{g_{0}%
}\widetilde{a}=\widetilde{a}\widetilde{v}_{g_{0}}$ and therefore
$\widetilde{v}_{g_{0}}^{\ast}\widetilde{a}=\widetilde{a}\widetilde{v}_{g_{0}%
}^{\ast}$ for all $g_{0}\in G,a\in A.$ Then,
\[
\widetilde{v}_{g_{0}}\widetilde{v}_{g_{0}}^{\ast},\widetilde{v}_{g_{0}}^{\ast
}\widetilde{v}_{g_{0}}\in A^{\prime}\cap\left(  M_{loc}(A)\times_{\alpha
}G\right)  .
\]
Standard calculations in $M_{loc}(A)\times_{\alpha}G$ show that
($\widetilde{v}_{g_{0}}\widetilde{v}_{g_{0}}^{\ast})=v(g_{0})^{\ast}%
v(g_{0})\delta_{e},$ and ($\widetilde{v}_{g_{0}}^{\ast}\widetilde{v}_{g_{0}%
})=\alpha_{g_{0}^{-1}}(v^{\ast}(g_{0})v(g_{0}))\delta_{e}$ so,
\[
\widetilde{v}_{g_{0}}\widetilde{v}_{g_{0}}^{\ast},\widetilde{v}_{g_{0}}^{\ast
}\widetilde{v}_{g_{0}}\in A^{\prime}\cap M_{loc}(A).
\]
Since A is prime, by Lemma 3 d) we have that $v(g_{0})v(g_{0})^{\ast}\in%
\mathbb{C}
1$ and $v(g_{0})^{\ast}v(g_{0})\in%
\mathbb{C}
1.$Therefore, there exists $\omega\in%
\mathbb{C}
,$ $\omega\neq0$ such that $\omega v(g_{0})$ is unitary. We can assume that
$\omega=1.$ By relation $(\ast$) above we hav%

\[
\alpha_{g_{0}}(a)=v^{\ast}(g_{0})av(g_{0}),a\in A.
\]
Hence $\alpha_{g_{0}}$ is inner in $M_{loc}(A),$\ so $g_{0}=e.$ Therefore
$v\in%
\mathbb{C}
1$ in $M_{loc}(A)\times_{\alpha}G=M_{loc}(A\times_{\alpha}G)$ and we are done.

Now, suppose that $\alpha$\ is strictly outer, i.e. $\left(  A\right)
^{\prime}\cap M_{loc}(A\times_{\alpha}G)=%
\mathbb{C}
I.$ Let $g_{0}\in G$ with $\alpha_{g_{0}}(a)=v_{0}av_{0}^{\ast}$ for a unitary
element $v_{0}\in M_{loc}(A).$ I will show that $g_{0}=e$ where $e$\ is the
unit of $G.$ If we denote $\widehat{v}_{0}=v_{0}\delta_{e}\in M_{loc}%
(A)\times_{\alpha}G$ $(=M_{loc}(A\times_{\alpha}G)$ according to Lemma 10 b))
and $\widetilde{b}=b\delta_{e}\in A\times_{\alpha}G\subset M_{loc}%
(A)\times_{\alpha}G$ for $b\in A,$\ then
\[
\widetilde{\alpha_{g_{0}}(a)}=\widehat{v}_{0}\widetilde{a}\widehat{v}%
_{0}^{\ast},a\in A.
\]
On the other hand, a standard calculation in $M_{loc}(A)\times_{\alpha}G$
gives
\[
\widetilde{\alpha_{g_{0}}(a)}=\delta_{g_{0}}\widetilde{a}\delta_{g_{0}^{-1}}.
\]
Therefore $\delta_{g_{0}^{-1}}\widehat{v}_{0}\in$ $\left(  A\right)  ^{\prime
}\cap M_{loc}(A\times_{\alpha}G)=%
\mathbb{C}
1,$ so $\widehat{v}_{0}=\mu\delta_{g_{0}}$ for some scalar $\mu.$\ Since the
support of $\widehat{v}_{0}\ $as a function on $G$\ is $\left\{  e\right\}
$\ and the support of $\delta_{g_{0}}$\ is $\left\{  g_{0}\right\}  $\ it
follows that $g_{0}=e.$ Hence every $\alpha_{g_{0}}$\ with g$_{0}\neq e$ \ is
properly outer.
\end{proof}

\bigskip

\begin{remark}
a) If $(A^{\alpha})^{\prime}\cap M_{loc}(A)=%
\mathbb{C}
I,$\ then every automorphism $\alpha_{t},t\neq e$\ is properly outer and
therefore, the action is strictly outer by the above theorem. Indeed if some
$\alpha_{g_{0}}$ \ is implemented by a unitary $u\in M_{loc}(A)$\ then,
clearly \ $u\in(A^{\alpha})^{\prime}\cap M_{loc}(A),$ so $g_{0}=e.$\newline b)
Suppose $A$\ and $A\times_{\alpha}G$\ are prime, $G$\ compact (in particular
finite) and $\alpha$\ faithful. If for every automorphism, $\beta$ \ of
$A$\ whose restriction to $A^{\alpha}$\ is the identity, there exists $g\in
G$\ such that $\beta=\alpha_{g},$\ then $(A^{\alpha})^{\prime}\cap M_{loc}(A)=%
\mathbb{C}
I.$ Indeed, if $A$\ and $A\times_{\alpha}G$\ are prime and $\alpha$\ is
faithful, by [8, Corollary 3.12] and [9, Lemma 2.2.] the hypotheses of [9,
Theorem 2.3.] are satisfied and the conclusion follows from that result.
\end{remark}

\bigskip In the next Section I will prove that, for finite abelian groups, the
condition $(A^{\alpha})^{\prime}\cap M_{loc}(A)=%
\mathbb{C}
I$\ is equivalent with the action being properly outer, and therefore strictly outer.

\subsection{Finite abelian groups}

\bigskip The proof of the following result is inspired by the proof of [5,
Theorem 4] for simple C*-algebras (see also [7, Proposition 8.10.13.]). The
hypothesis that $A$\ is simple is essential in their result.

\begin{proposition}
Let $(A,G,\alpha)$ be a C*-dynamical system with $G$\ finite abelian, $\alpha$
faithful and $A$\ prime. Then $\left(  A^{\alpha}\right)  ^{\prime}\cap
M_{loc}(A)=%
\mathbb{C}
I$\ if and only if for every $t\neq0,$ $\alpha_{t}$ is a properly outer automorphism.
\end{proposition}

\begin{proof}
Suppose $\left(  A^{\alpha}\right)  ^{\prime}\cap M_{loc}(A)=%
\mathbb{C}
I$. If $\alpha_{t_{0}}$ is implemented by a unitary $u_{t_{0}}$\ in
$M_{loc}(A)$, then, clearly, $u_{t_{9}}\in$ $\left(  A^{\alpha}\right)
^{\prime}\cap M_{loc}(A)=%
\mathbb{C}
I.$ Conversely suppose that $\alpha_{t}$\ is properly outer for every $t\in
G,t\neq0.$ According to Theorem 9, the center of $M_{loc}(A)^{\alpha}$\ is trivial.

Let $B=\left(  A^{\alpha}\right)  ^{\prime}\cap M_{loc}(A).$ Since the center
of $M_{loc}(A)^{\alpha}$ is trivial, from Lemma 3 b) it follows that the fixed
point algebra of the action $\alpha$\ restricted to the $\alpha$-invariant
subalgebra $B$ is trivial. Suppose that $B\neq%
\mathbb{C}
I.$ Then, there exists $\gamma_{0}\in\widehat{G}$\ such that $B_{\gamma_{0}%
}\neq\left\{  0\right\}  .$ If $x,y\in B_{\gamma_{0}},$ then $x^{\ast}y\in%
\mathbb{C}
I$ \ (the fixed point algebra of the restriction of $\alpha$\ to $B$).
Therefore there exists a unitary, $u\in B_{\gamma_{0}}$, such that
$B_{\gamma_{0}}=%
\mathbb{C}
u.$ Since $u$\ is unitary, it follows that $n\gamma_{0}\in Sp(\alpha|_{B})$
for every $n\in%
\mathbb{Z}
,$ so $B_{n\gamma_{0}}=%
\mathbb{C}
u^{n}$ for every $n\in%
\mathbb{Z}
.$

Let $H=\left\{  t\in G:\left\langle t,\gamma_{0}\right\rangle =1\right\}  .$
Then, since $\alpha_{t}(u)=\left\langle t,\gamma_{0}\right\rangle u$ for every
$t\in G,$\ it follows that $\alpha_{t}(u)=u$ for all $t\in H.$ Notice that the
annihilator, $H^{\perp},$ of $H$\ in $\widehat{G}$ equals $%
\mathbb{Z}
\gamma_{0}.$ Now, consider the C*-subalgebra
\[
M_{loc}(A)^{H}=\left\{  x\in M_{loc}(A):\alpha_{t}(x)=x\text{ for all }t\in
H\right\}  .
\]
By [7, 8.10.3.], if $\widetilde{\alpha}$ is the action of $G/H$ on
$M_{loc}(A)^{H},$ we have $Sp(\widetilde{\alpha})\subset%
\mathbb{Z}
\gamma_{0}.$ Therefore, the spectral subspaces $(M_{loc}(A)^{H})_{n\gamma_{0}%
}$ have dense span in $M_{loc}(A)^{H}.$ Next, we will prove that $u$\ commutes
with $(M_{loc}(A)^{H})_{n\gamma_{0}},n\in%
\mathbb{Z}
$ and thus $u$\ belongs to the center of $M_{loc}(A)^{H}.$ Let $x\in
(M_{loc}(A)^{H})_{n\gamma_{0}}\subset(M_{loc}(A))_{n\gamma_{0}}$\ Then
\[
x=x(u^{\ast})^{n}u^{n}\in(M_{loc}(A))^{\alpha}u^{n}.
\]
On the other hand, since $u\in B=(A^{\alpha})^{\prime}\cap M_{loc}(A),$ in
particular $u\in(A^{\alpha})^{\prime}.$\ applying Lemma 3 b) again,\ it
follows that $u$\ commutes with $M_{loc}(A)^{\alpha}$ , so $u$ commutes with
$x.$ Hence $u$\ belongs to the center of $M_{loc}(A)^{H}.$ Thus the center of
$M_{loc}(A)^{H}$ contains a non scalar element. Since, by hypothesis, $A$ is
prime, in particular $H$-prime,\ applying Theorem 9 for $G=H$, it follows that
there exists $t_{0}\in H$\ such that $\alpha_{t_{0}}$ is implemented by a
unitary in $M_{loc}(A)^{H},$ thus $\alpha_{t_{0}}$\ is not properly outer,
which is a contradiction.
\end{proof}

\bigskip

As stated before, all the results obtained in this paper, hold for not
necessarily separable C*-algebras. The next theorem was considered in [2] for
compact groups and separable C*-algebras. In addition to separability, in the
equivalence of a) and b) below the authors make the assumption that
$A^{\alpha}$ is prime.

\bigskip

\begin{theorem}
Let $(A,G,\alpha)$\ be a C*-dynamical system with $A$\ prime, $G$\ finite
abelian and $\alpha$\ faithful. The following conditions are
equivalent\newline a) $\alpha_{g}$ is properly outer for every $g\in
G,g\neq0.$\newline b) $(A^{\alpha})^{\prime}\cap M_{loc}(A)=%
\mathbb{C}
I.$\newline c) $(A)^{\prime}\cap M_{loc}(A\times_{\alpha}G)=%
\mathbb{C}
I.$\newline d) $A^{\alpha}$\ is prime and $\widehat{\alpha}_{\gamma}$\ is
properly outer (in $M_{loc}(A\times_{\alpha}G))$ for every $\gamma
\in\widehat{G},\gamma\neq0.$
\end{theorem}

\begin{proof}
\bigskip The equivalence of a) and b) follows from Proposition 13. The
equivalence of a) and c) follows from Theorem 11 in the particular case of
abelian groups.

c)$\Rightarrow$d)$,$ Notice first that from c) it follows that the center of
$M_{loc}(A\times_{\alpha}G)$\ consists of scalars. Therefore, by Lemma 1,
$A\times_{\alpha}G$\ is prime, so, by [8, Corollkary 3.13], $A^{\alpha}$\ is
prime. If for some $\gamma\in\widehat{G},\widehat{\alpha}_{\gamma}$ is
implemented by a unitary $u$ in $M_{loc}(A\times_{\alpha}G),$ then, as
$(A\times_{\alpha}G)^{\widehat{\alpha}}=A,$\ it follows that $u\in(A)^{\prime
}\cap M_{loc}(A\times_{\alpha}G)=%
\mathbb{C}
I.$ Hence $\gamma=0.$

d)$\Rightarrow$c) If $A^{\alpha}$\ is prime and $\alpha$\ is faithful, from
property 5) in the list of properties of $Sp(\alpha)$ discussed after
Definition 6, Section 2, it follows that $Sp(\alpha)=\widehat{G}.$\ Applying
[8 Cor 3.13.] it follows that $A\times_{\alpha}G$ is prime. Therefore
$A\times_{\alpha}G$\ is prime and $\widehat{\alpha}_{\gamma}$\ is properly
outer (in $M_{loc}(A\times_{\alpha}G))$ for every $\gamma\in\widehat{G}%
,\gamma\neq0.$ Since $(A\times_{\alpha}G)^{\widehat{\alpha}}=A,$ from
Proposition 13 applied to the C*-dynamical system $(A\times_{\alpha
}G,\widehat{G},\widehat{\alpha})$ (instead of $(A,G,\alpha)$) it follows that
$(A)^{\prime}\cap M_{loc}(A\times_{\alpha}G)=%
\mathbb{C}
I.$
\end{proof}

\bigskip

\begin{center}
{\LARGE References}
\end{center}

\bigskip

[1] P. Ara and M.Mathieu, \textit{Local Multipliers of C*-Algebras, }Springer, 2003.

[2] O. Bratteli, G. A. Elliott, D. E. Evans and A. Kishimoto,
\textit{Quasi-Product Actions of a Compact Abelian Group on a C*-Algebra,
}Tohoku J. of Math., 41 (1989), 133-161.

[3] R. Dumitru and C. Peligrad, \textit{Compact Quantum Actions on C*-Algebras
and Invariant Derivations, }Proceedings of the Amer. Math. Soc., vol. 135,
Nr.12 (2007), 3977-3984.

[4] D. Olesen, \textit{Inner *-Automorphisms of Simple C*-Algebras, }Commun.
math. Phys. 44, (1975), 175---190.

[5] D. Olesen, G. K. Pedersen and E. Stormer, \textit{Compact Abelian Groups
of Automorphisms of Simple C*-Algebras, }lnventiones Math. 39 (1977), 55--64.

[6] G. K. Pedersen, \textit{Approximating Derivations on Ideals of
C*-Algebras}, Inventiones Math. 45 (1978), 299-305.

[7] G. K. Pedersen, \textit{C*-Algebras and Their Automorphism Groups, }Second
edition, Academic Press, 2018.

[8] C. Peligrad, \textit{Locally Compact Group Actions on C*-Algebras and
Compact Subgroups, }J. Functional Anal. 76, Nr. 1 (1988), 126-139.

[9] C. Peligrad, \textit{Duality for Compact Group Actions on Operator
Algebras and Applications: Irreducible Inclusions and Galois Correspondence,}
Internat. J. Math. 31, Nr. 9 (2020), 12 pp.

[10] M. Rieffel, \textit{Actions of Finite Groups on C*-Algebras, }Math.
Scand., 47 (1980), 157-176.

[11] S. Vaes, \textit{The Unitary Implementation of a Locally Compact Group
Action, }J. Functional Anal. 180 (2001), 426-480.
\end{document}